\documentclass[11pt]{amsart}
\pdfoutput=1

\title{The number of topological types of trees}
\author{Thilo Krill}
\author{Max Pitz}
\address{Universit\"at Hamburg, Department of Mathematics, Bundesstrasse 55 (Geomatikum), 20146 Hamburg, Germany}
\email{\{thilo.krill, max.pitz\}@uni-hamburg.de}

\usepackage{amsmath,amssymb,amsthm}
\usepackage{mathtools}
\usepackage{tikz}
\usepackage{comment}
\usepackage{enumitem}

\usepackage{xcolor}
\usepackage[unicode]{hyperref}
\hypersetup{
 colorlinks,
 linkcolor={red!60!black},
 citecolor={green!60!black},
 urlcolor={blue!60!black},
}
\usepackage[abbrev, msc-links]{amsrefs}

\usepackage[utf8]{inputenc}
\usepackage[T1]{fontenc}
\usepackage{lmodern}
\usepackage[babel]{microtype}
\usepackage[english]{babel}

\let\polishlcross=\l
\def\l{\ifmmode\ell\else\polishlcross\fi}

\let\theta=\vartheta
\let\rho=\varrho
\let\phi=\varphi

\def\NN{\mathbb N}

\def\cA{{\mathcal A}}

\def\cC{{\mathcal C}}

\def\cP{{\mathcal P}}

\def\cT{{\mathcal T}}

\newcommand{\parentheses}[1]{{\left( {#1} \right)}}

\newcommand{\p}{\parentheses}

\newcommand{\Set}[1]{{\left\lbrace {#1} \right\rbrace}}

\def\set#1:#2{\Set{{#1} \colon {#2}}}

\renewcommand{\triangleleft}{\trianglelefteq}
\newcommand{\nottriangleleft}{\not\kern-1pt\mathrel{\triangleleft}}

\newcommand{\upcl}[1]{\lfloor #1 \rfloor}

\theoremstyle{plain}
\newtheorem{thm}{Theorem}[section]
\newtheorem*{thmmain}{Theorem}

\newtheorem{cor}[thm]{Corollary}
\newtheorem{lemma}[thm]{Lemma}

\theoremstyle{definition}

\newtheorem{defn}[thm]{Definition}

\DeclareMathOperator{\successor}{succ}

\DeclareMathOperator{\rank}{rank}

\begin{document}

\begin{abstract}
Two graphs are of the same \emph{topological type} if they can be mutually embedded into each other topologically.
We show that there are exactly $\aleph_1$ distinct topological types of countable trees. In general, for any infinite cardinal $\kappa$ there are exactly $\kappa^+$ distinct topological types of trees of size $\kappa$. This solves a problem of van der Holst from 2005.
\end{abstract}


\vspace*{-1cm}

\maketitle

\section{Introduction}

A graph-theoretic tree $T$ is a \emph{topological minor} of another tree $S$, written $T \leq S$, if some subdivision of $T$ embeds as a subgraph into $S$. Nash-Williams \cite{nash1965well} proved in 1965 the seminal result that the class of graph-theoretic trees is \emph{well-quasi-ordered} under $\leq$, i.e.\ that it is a reflexive and transitive relation without infinite strictly decreasing sequences or infinite antichains.

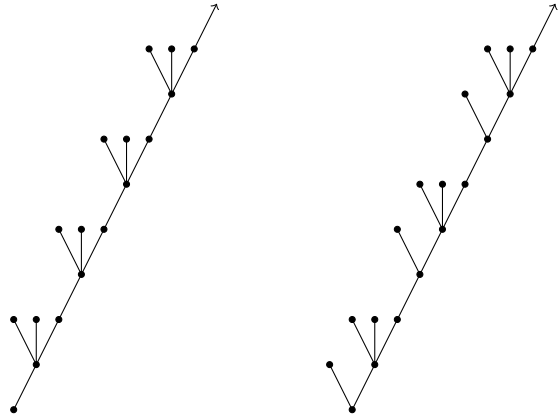
\begin{figure}[h]
\begin{center}
    \begin{tikzpicture}[scale=0.3]
    \foreach \i in {1,...,8}
    {
    \draw[fill=black]  (\i,{2*\i}) circle (.8ex);
    \draw  (\i,{2*\i}) -- (\i+1,{2*\i+2});
    };
    \draw[->]  (9,18) -- (10,20);
     \draw[fill=black]  (9,18) circle (.8ex);
     \foreach \i in {3,5,7,9}
    {
    \draw[fill=black]  ({\i-2},{2*\i}) circle (.8ex);
    \draw ({\i-2},{2*\i}) -- ({\i-1},{2*\i-2});
    };
      \foreach \i in {3,5,7,9}
    {
    \draw[fill=black]  ({\i-1},{2*\i}) circle (.8ex);
    \draw ({\i-1},{2*\i}) -- ({\i-1},{2*\i-2});
    };
    
    \begin{scope}[xshift=15cm]
   \foreach \i in {1,...,8}
    {
    \draw[fill=black]  (\i,{2*\i}) circle (.8ex);
    \draw  (\i,{2*\i}) -- (\i+1,{2*\i+2});
    };
       \draw[fill=black]  (9,18) circle (.8ex);
    \draw[->]  (9,18) -- (10,20);
     \foreach \i in {2,3,5,6,8,9}
    {
    \draw[fill=black]  ({\i-2},{2*\i}) circle (.8ex);
    \draw ({\i-2},{2*\i}) -- ({\i-1},{2*\i-2});
    };
       \foreach \i in {3,6,9}
    {
    \draw[fill=black]  ({\i-1},{2*\i}) circle (.8ex);
    \draw ({\i-1},{2*\i}) -- ({\i-1},{2*\i-2});
    };
    \end{scope}
    \end{tikzpicture}
\end{center}
\label{fig_intro}
\caption{Distinct trees of the same topological type.}
\end{figure}

However, this embedding relation $\leq$ is not anti-symmetric: Two distinct trees $T$ and $S$ may well be topological minors of each other, i.e.\ $T \leq S$ and $S \leq T$. In this case, we say they are of the same \emph{topological type}, written $T \equiv S$.
Describing the hierarchy of graph-theoretic trees under the quasi-ordering  $\leq$ means understanding the partial order that $\leq$ induces on the topological types of trees. By Nash-Williams's theorem, this is a well-partial-order. But determining its most fundamental characteristic, namely its cardinality, has been an open problem until now.

Indeed, while up to isomorphism there are exactly $2^{\aleph_0}$ countable trees, determining the exact number of topological types of countable trees is an open problem posed by van der Holst in 2005 (see \cite{matthiesen2006there}). Building on the trees in Figure~\ref{fig_intro}, we get examples of $2^{\aleph_0}$ non-isomorphic trees of the same topological type.
So a priori, it would have been conceivable that there are only countably many topological types of countable trees. However, Matthiesen \cite{matthiesen2006there} showed in 2006, by an indirect proof building on Nash-Williams's theorem, that there are uncountably many topological types of countable trees. Bruno \cite{bruno2017family} in 2017  gave an explicit construction of uncountably many topological types of subtrees of the binary tree. Recently, Bruno and Szeptycki \cite{bruno2022there} gave the first indication that this bound could be sharp by establishing  that there are exactly $\aleph_1$ many topological types of locally finite trees with only countably many rays.

Our main result confirms this pattern for all trees and all cardinalities:

\begin{thmmain}
    For any infinite cardinal $\kappa$ there are exactly $\kappa^+$ distinct topological types of trees of size $\kappa$. 
\end{thmmain}

Our proof uses a new rank function for trees inspired by Nash-Williams's work on the better-quasi-ordering of infinite trees.

\section{The lower bound}

For common concepts in graph theory and set theory see the textbooks by Diestel \cite{diestel2015book} and Jech \cite{jech2013set}. We write $\kappa^+$ for the successor cardinal of $\kappa$, and $\alpha+1$ for the successor ordinal of $\alpha$.

We recall Schmidt's rank function for rayless graphs \cite{schmidt1983ordnungsbegriff}, see also \cite{halin1998structure} for an English account:
We say that a graph $G$ has \emph{rank} 0 if $G$ is finite. Given an ordinal $\alpha>0$, we assign \emph{rank} $\alpha$ to $G$ if $G$ does not already have a rank $<\alpha$ and there exists a finite set of vertices $X$ in $G$ such that all components of $G - X$ have a rank $<\alpha$.
Schmidt proved that a graph has a rank if and only if it is rayless. Moreover, a routine induction on the rank shows that Schmidt's rank function is non-decreasing with respect to the (topological) minor relation, see e.g.\ \cite[Proposition~4.4]{halin1998structure}:

\begin{lemma}\label{monotony_1}
If a rayless graph $H$ is a (topological) minor of a rayless graph $G$, then the rank of $H$ is at most the rank of $G$. \hfill \qed
\end{lemma}

We can now give the argument for the lower bound in our main theorem:

\begin{lemma}
\label{lem_minortwinclasses}
   For any infinite cardinal $\kappa$ there are at least $\kappa^+$ distinct topological types of (rayless) trees of size $\kappa$. 
\end{lemma}

\begin{proof}
We show that for all ordinals $0<\alpha<\kappa^+$, there exists a rayless tree $T_\alpha$ of size $\kappa$ and rank $\alpha$.
Then it follows from Lemma~\ref{monotony_1} that all $T_\alpha$ belong to different topological types, establishing the assertion of the lemma.

We construct the $T_\alpha$ by recursion on $\alpha$, beginning with $T_1$ as the $\kappa$-star.
For successor steps, take countably many disjoint copies $T_n$ ($n \in \NN$) of $T_\alpha$ and obtain $T_{\alpha+1}$ by adding a new vertex $v$ to $\bigsqcup_{n \in \NN} T_n$ and connecting it to the root of every $T_n$. Then deleting $\{v\}$ witnesses that $T_{\alpha+1}$ has rank at most $\alpha+1$. On the other hand, every finite set of vertices $X$ of $T_{\alpha+1}$ avoids infinitely many copies of $T_\alpha$, so there are components of $T_{\alpha+1} - X$ containing copies of $T_\alpha$. Every such component has rank at least $\alpha$ by Lemma \ref{monotony_1}, showing that $T_{\alpha+1}$ has rank at least $\alpha+1$.

For limit steps, obtain $T_\ell$  by adding a new vertex $v$ to $\bigsqcup_{\alpha<\ell}T_\alpha$ and connecting it to the root of every $T_\alpha$ for $\alpha < \ell$.
Then deleting $\{v\}$ witnesses that $T_{\ell}$ has rank at most $\ell$; on the other hand, every finite set of vertices $X$ of $T_{\ell}$ avoids almost all $T_\alpha$ copies for $\alpha < \ell$, so $T_{\ell} - X$ contains components of arbitrarily large rank below $\ell$ by Lemma \ref{monotony_1}, showing that $T_{\ell}$ has rank at least $\ell$.
\end{proof}

\section{Better-quasi-orderings}\label{sec_bqo}

A \emph{quasi-ordering} is a binary relation that is reflexive and transitive.
A quasi-ordering $\leq$ on set $Q$ is a \emph{well-quasi-ordering} if for every sequence $q_1,q_2,q_3, \ldots$ of elements in $Q$ there are indices $n < m \in \mathbb{N}$ such that $q_n \leq q_m$.
We define an equivalence relation $\equiv$ on $Q$ by $q\equiv q'$ if both $q\leq q'$ and $q'\leq q$.
We abbreviate $|Q|_\equiv:=|Q/{\equiv}|$.

Let $Q$ be quasi-ordered and $\kappa$ an infinite cardinal.
We say that $q \in Q$ is \emph{$\kappa$-embeddable in $Q$} if there exist at least $\kappa$ many elements $q' \in Q$ with $q \leq q'$.
We need the following routine result, a proof of which can be found e.g.\ in \cite[Lemma~3.3]{bowler2018ubiquity}:

\begin{lemma}
\label{exceptional_set}
For any well-quasi-order $Q$ and infinite cardinal $\kappa$, the number of elements of $Q$ which are not $\kappa$-embeddable in $Q$ is less than $\kappa$. \hfill \qed
\end{lemma}

Let $(Q,\leq_Q)$ be a quasi-order. Following Nash-Williams \cite{nash1965well}, we consider the quasi-ordering on the power set $\cP(Q)$ where for $A,B \subseteq Q$ we have $A \leq B$ if there is an injective function $f \colon A \to B$ such that $a \leq_Q f(a)$ for all $a \in A$.
Recall that $\cP(Q)$ is not necessarily well-quasi-ordered if $Q$ is well-quasi-ordered (see \cite{rado1954partial}).
This is remedied by the introduction of the concept of a \emph{better-quasi-ordering}. We shall not define this concept precisely; we only use as a blackbox that every better-quasi-ordered set is also well-quasi-ordered, that $\cP(Q)$ is better-quasi-ordered if $Q$ is better-quasi-ordered, and that the class of all trees is better-quasi-ordered under the topological minor relation \cite{nash1965well} (also see \cite{laver1978better}, \cite{kuhn2001revisited}).

We write $\cP_\kappa(Q)$ for the set of subsets of $Q$ of size exactly $\kappa$ and $\cP_{\leq \kappa}(Q)$ for the set of subsets of $Q$ of size at most $\kappa$.
Extending an idea from \cite{bruno2022there}, we prove the following result on the number of equivalence classes in $\cP_\kappa(Q)$:

\begin{lemma}
\label{bqo_sequences_inj}
Let $\mu$ be an infinite cardinal and $Q$ a better-quasi-ordered set with $|Q|_\equiv=\mu$.
Then $|\cP_\kappa(Q)|_\equiv=\mu$ for all cardinals $\kappa < \aleph_{\mu^+}$.
\end{lemma}

\begin{proof}
By induction on $\kappa$.
Suppose for a contradiction that $|\cP_\kappa(Q)|_\equiv\geq\mu^+$.
By the Erd\H{o}s-Dushnik-Miller theorem, every partial order $(P,\leq)$ contains an infinite antichain or a chain of size $|P|$, see \cite[Theorem~5.25]{dushnik1941partially}.
As $(Q,\leq_Q)$ is better-quasi-ordered, $(\cP_\kappa(Q), \leq)$ is well-quasi-ordered \cite[Corollary 28A]{nash1965well}. So the partial order $\cP_\kappa(Q)/\equiv$ contains no infinite antichains and thus contains a chain of size $\mu^+$. Since $\cP_\kappa(Q)/\equiv$ is well-founded, this chain is well-ordered. Hence, there is strictly increasing chain $\cA = (A_k\colon k < \mu^+)$ in $\cP_\kappa(Q)$. 

By applying Lemma \ref{exceptional_set} to each induced suborder $(A_k,\leq)$ of $(Q,\leq_Q)$, we obtain for every $A_k \in \cA$ a subset $X_k \subseteq A_k$ with $|X_k| < \kappa$ such that all elements of $A_k \setminus X_k$ are $\kappa$-embeddable in $(A_k,\leq)$.
Since $|X_k| < \kappa < \aleph_{\mu^+}$ for all $k < \mu^+$, there are at most $\mu$ different possible cardinalities for the sets $X_k$.
Since $\mu^+$ is regular, we may assume without loss of generality that $|X_k| = \nu$ for all $k < \mu^+$ and some cardinal $\nu < \kappa$.
Furthermore, we may assume that all sets $X_k$ for $k < \mu^+$ are pairwise equivalent with respect to $\equiv$ since $|\cP_\nu(Q)|_\equiv = \mu$ by induction.

Next, let $\set{q_\ell}:{\ell < \mu}$ be a representation system for the equivalence classes of $Q/{\equiv}$. 
For every $q_\ell$ that is $\kappa$-embeddable in some $A \in \cA$, we pick a suitable $A_{k(\ell)} \in \cA$ witnessing this. Let $k^* := \sup \set{k(\ell)}:{\ell < \mu}  < \mu^+$.
We arrive at the desired contradiction once we have proved that $A_k \equiv A_{k^*}$ for all $k > k^*$.
Since $X_{k} \equiv X_{{k^*}}$ already, it suffices to show that
$$A_k \setminus X_{k} \leq A_{k^*} \setminus X_{{k^*}}$$
for all $k > k^*$.
For this, we need an injective function $f: A_k \setminus X_{k} \to A_{k^*} \setminus X_{{k^*}}$ that satisfies $a \leq_Q f(a)$ for all $a \in A_k \setminus X_{k}$.
Enumerate $A_k \setminus X_{k} = \{a_i \colon i < \kappa\}$, let $i < \kappa$, and suppose that $f$ has been defined on $a_j$ for all $j < i$.
Since $a_i$ is $\kappa$-embeddable in $A_k$, it is also $\kappa$-embeddable in $A_{k'}$ for some $k' \leq k^*$ by the definition of $k^*$.
Since $A_{k'} \leq A_{k^*}$, the element $a_i$ is also $\kappa$-embeddable in $A_{k^*}$. Hence we can find an element $b \in A_{k^*} \setminus X_{{k^*}}$ such that $a_i \leq_Q b$ and $b$ is distinct from all values of $f$ that have already been defined.
We set $f(a_i) := b$, which completes the construction of $f$.
\end{proof}

\begin{cor}
\label{cor_sequences_inj}
Let $\mu$ be an infinite cardinal and $Q$ a better-quasi-ordered set with $|Q|_\equiv=\mu$.
Then $|\cP_{\leq\kappa}(Q)|_\equiv=\mu$ for all cardinals $\kappa < \aleph_{\mu^+}$.
\end{cor}

\begin{proof}
Since $\kappa < \aleph_{\mu^+}$, there exist at most $\mu$ cardinals $\leq \kappa$.
Hence $|\cP_{\leq\kappa}(Q)|_\equiv \leq \mu * \mu = \mu$ by Theorem \ref{bqo_sequences_inj} applied to $\cP_\nu(Q)$ for all cardinals $\nu \leq \kappa$.
\end{proof}

\section{The upper bound}

We consider rooted, graph theoretic trees and tree-order preserving topological minors. For this, we introduce a minimal amount of notation, cf.~\cite[\S12.2]{diestel2015book}.
Recall that fixing a root $r$ of a graph-theoretic tree $T$ yields a natural tree-order $\leq_r$ on $T$ where $t \leq_r s$ if $t$ lies on the unique path from $r$ to $s$ in $T$. Given a rooted tree, write $\upcl{t}$ for the subtree of $T$ induced by the set $\set{t' \in T}:{t\le_r t'}$ with root $t$. The neighbours of $t$ in $\upcl{t}$ are the \emph{successors} of $t$, denoted by the set $\successor(t)$.
Given rooted trees $T$ and $S$, we write $T \leq S$ if there exists a topological minor embedding $\varphi \colon T \to S$ that preserves the tree-order: If $x \leq y$ in $T$ then $\varphi(x) \leq \varphi(y)$ in $S$.

We now introduce a new rank function inspired by the proof methods of the better-quasi-ordering of trees due to Nash-Williams:

\begin{defn}
We say that a tree $T$ has rank 0 if $\upcl{t} \equiv T$ holds for all $t \in T$.
Given an ordinal $\alpha > 0$, we assign rank $\alpha$ to $T$ if $T$ does not already have a rank $<\alpha$ and for all $t \in T$, we have either $\upcl{t} \equiv T$ or $\upcl{t}$ has rank $<\alpha$. We also write $\rank(T)$ for the rank of $T$.
\end{defn}

\begin{lemma}\label{lem_trees_have_ranks}
Every tree of size at most $\kappa$ has a rank $<\kappa^+$.
\end{lemma}

\begin{proof}
Suppose for a contradiction that there is a tree in $\cT$ which does not have a rank $<\kappa^+$. Since rooted trees are well-quasi-ordered under $\leq$ by Nash-Williams's theorem \cite{nash1965well}, there exists a $\leq$-minimal such tree $T$.
Then for every $t \in T$ with $\upcl{t} \not \equiv T$ we have $\rank (\upcl{t}) <\kappa^+$ by minimality of $T$.
However, the rank of $T$ is at most
$$\sup \{\rank(\upcl{t})\colon t \in T, \upcl{t} \not\equiv T\} + 1,$$
which is an ordinal $<\kappa^+$ since $|T| \leq \kappa$.
This contradicts the choice of $T$.
\end{proof}

For the remainder of this section, let $\kappa$ be a fixed infinite cardinal. We write $\cT$ for the class of rooted trees of size at most $\kappa$, and $\cC$ for the set of cardinals of size at most $\kappa$. 

Let $T$ be a tree in $\cT$. For all $t \in T$, let
$$\Gamma(t) := \p{|\{s \in \successor(t)\colon \upcl{s} \equiv T\}|,
\{\upcl{s}\colon s \in \successor(t), \upcl{s} \not\equiv T\}\}} \in \cC \times \cP_{\leq\kappa}(\cT).$$
Furthermore, we define
$$\Theta(T) := \{\Gamma(t) \colon t \in T\} \in \cP_{\leq\kappa}(\cC \times \cP_{\leq\kappa}(\cT)).$$

Given two quasi-orderings $(Q,\leq)$ and $(R,\leq)$, we define a quasi-ordering on $Q \times R$ by letting $(q,r) \leq (q',r')$ if $q \leq q'$ and $r \leq r'$.
Together with the quasi-ordering on $\cP(Q)$ defined in Section \ref{sec_bqo}, this yields a quasi-ordering on the set $\cP_{\leq\kappa}(\cC \times \cP_{\leq\kappa}(\cT))$ considered in the definition of $\Theta(T)$ above.
Nash-Williams showed in \cite[Lemma 29]{nash1965well}:

\begin{lemma}
\label{lem_nash-williams}
For all rooted trees with $\Theta(T) \leq \Theta(S)$, we have $T\leq S$. \hfill \qed
\end{lemma}

Finally, we give the argument for the upper bound in our main theorem, in a stronger version for rooted trees:

\begin{thm}
For any infinite cardinal $\kappa$ there are at most $\kappa^+$ distinct topological types of rooted trees of size $\kappa$. 
\end{thm}

\begin{proof}
For all ordinals $\alpha < \kappa^+$, we write $\cT_\alpha$ for the class of all trees of size at most $\kappa$ and rank $\alpha$ and $\cT_{<\alpha}$ for the class of all trees of size at most $\kappa$ and rank $<\alpha$.
We show by induction on $\alpha$ that $|\cT_\alpha|_\equiv \leq \kappa$ holds for all $\alpha < \kappa^+$.
Then it follows from Lemma \ref{lem_trees_have_ranks} that
$$|\cT|_\equiv = \left\vert\bigcup_{\alpha<\kappa^+}\cT_\alpha \right\vert_\equiv\leq\kappa^+,$$
completing the proof.

Let $\alpha < \kappa^+$ and suppose $|\cT_\beta|_\equiv \leq \kappa$ for all $\beta < \alpha$.
Consider the function
$$\cT_\alpha \to \cP_{\leq\kappa}(\cC \times \cP_{\leq\kappa}(\cT_{<\alpha})),\, T \mapsto \Theta(T).$$
If $T, S \in \cT_\alpha$ belong to different twin classes, then also $\Theta(T)$ and $\Theta(S)$ belong to different twin classes of $\cP_{\leq\kappa}(\cC \times \cP_{\leq\kappa}(\cT_{<\alpha}))$ by Lemma \ref{lem_nash-williams}.
We conclude that
$$|\cT_\alpha|_\equiv \leq |\cP_{\leq\kappa}(\cC \times \cP_{\leq\kappa}(\cT_{<\alpha}))|_\equiv.$$
Thus it suffices to show
$$
|\cP_{\leq\kappa}(\cC \times \cP_{\leq\kappa}(\cT_{<\alpha}))|_\equiv \leq \kappa.    
$$
First, we argue that $|\cT_{<\alpha}|_\equiv \leq \kappa$: This is clear if $\alpha = 0$.
If $\alpha > 0$, we have $|\cT_\beta|_\equiv \leq \kappa$ for all $\beta < \alpha$ and hence $|\cT_{<\alpha}|_\equiv = |\bigcup_{\beta < \alpha}\cT_\beta|_\equiv \leq \kappa$ since $\alpha < \kappa^+$.
Next, $\cT$ and therefore $\cT_{<\alpha}$ is better-quasi-ordered by \cite{nash1965well} and thus we have
$|\cP_{\leq\kappa}(\cT_{<\alpha})|_\equiv \leq \kappa$ by Corollary \ref{cor_sequences_inj}.
Then it follows from cardinal arithmetic that also $|\cC \times \cP_{\leq\kappa}(\cT_{<\alpha})|_\equiv \leq \kappa$.
Finally, since $\cP_{\leq\kappa}(\cT_{<\alpha})$ is better-quasi-ordered by \cite[Corollary 28A]{nash1965well} and hence $\cC \times \cP_{\leq\kappa}(\cT_{<\alpha})$ is better-quasi-ordered by \cite[Corollary 22A]{nash1965well}, applying Corollary \ref{cor_sequences_inj} once more yields  
$|\cP_{\leq\kappa}(\cC \times \cP_{\leq\kappa}(\cT_{<\alpha}))|_\equiv \leq \kappa$.
\end{proof}

\bibliographystyle{plain}
\bibliography{ref}

\end{document}